\documentclass[12pt]{amsart}
\usepackage{amscd,amssymb}
\usepackage{amsthm,amsmath,amssymb}
\usepackage[matrix,arrow]{xy}
\usepackage{longtable}
\usepackage{comment}
\usepackage{float}
\usepackage{graphicx}
\usepackage{booktabs}

\newtheorem{example}[equation]{Example}
\newtheorem*{example*}{Example}
\newtheorem{definition}[equation]{Definition}

\newtheorem{corollary}[equation]{Corollary}
\newtheorem{proposition}[equation]{Proposition}
\newtheorem{conjecture}[equation]{Conjecture}
\newtheorem*{conjecture*}{Conjecture}

\newtheorem*{question*}{Question}

\newtheorem*{problem*}{Problem}
\newtheorem*{theorem*}{Theorem}

\theoremstyle{remark}
\newtheorem{remark}[equation]{Remark}
\newtheorem*{remark*}{Remark}

\newtheorem*{ack}{Acknowledgements}

\newcommand{\vol}{{\operatorname{Vol}}}
\newcommand{\lct}{{\operatorname{lct}}}

\makeatletter\@addtoreset{equation}{section} \makeatother

\usepackage{hyperref}
\hypersetup{colorlinks,linkcolor={blue},citecolor={red}}

\newcommand{\bQ}{\mathbb{Q}}

\newcommand{\rd}{\mathrm{d}}

\title{On the Cheltsov--Rubinstein conjecture}

\begin{document}


\author[K.~Fujita]{Kento Fujita}
\address{Department of Mathematics, Graduate School of Science, Osaka University, Toyonaka,
Osaka 560-0043, Japan.}
\email{fujita@math.sci.osaka-u.ac.jp}
\author[Y.~Liu]{Yuchen Liu}
\address{Department of Mathematics, Yale University, New Haven, CT 06511, USA.}
\email{yuchen.liu@yale.edu}
\author[H.~S\"u{\ss}]{Hendrik S\"u{\ss}}
\address{School of Mathematics, The University of Manchester, Alan Turing
Building, Oxford Road Manchester M13 9PL.}
\email{hendrik.suess@manchester.ac.uk}
\author[K.~Zhang]{Kewei Zhang}
\address{Beijing International Center for Mathematical Research, Peking University, No. 5 Yiheyuan Road Haidian District, Beijing, China.}
\email{kwzhang@pku.edu.cn}
\author[Z.~Zhuang]{Ziquan Zhuang}
\address{Department of Mathematics, Princeton University, Princeton, NJ, 08544-1000.}
\email{zzhuang@math.princeton.edu}

\begin{abstract}
In this note we investigate the Cheltsov--Rubinstein conjecture. We show that this conjecture does not hold in general and some counterexamples will be presented.
\end{abstract}

\maketitle

\section{Introduction}
\label{section:intro}
In the study of canonical metrics on Fano type manifolds,
K\"ahler-Einstein edge (KEE) metrics are a natural generalization of K\"ahler-Einstein metrics: they are smooth metrics on the complement of a divisor, and have a conical singularity of angle $2\pi\beta$ transverse to that complex edge (see \cite{R14} for a survey, precise definition and references). Considerable amount of work on KEE metrics in recent years has concerned the behavior of such metrics when the cone angle is relatively large (e.g., close to $2\pi$). 

In 2013, Cheltsov--Rubinstein \cite{CR15} initiated a systematic study of the behavior in the other extreme when the \emph{cone angle $\beta$ goes to zero}. To explore this small cone angle world, it is natural to work on \emph{asymptotically log Fano varieties}, a class of varieties introduced in op. cit.

\begin{definition}{\cite{CR15}}
\label{definition:asymp-log-Fano}
Let $X$ be a normal projective variety over $\mathbb{C}$. Let $D=\sum D_i$ be an effective divisor on $X$, where each $D_i$ is a prime divisor. We say the pair $(X,D)$ is (strongly) asymptotically log Fano if the log pair $(X,\sum(1-\beta_i)D_i)$ is log Fano for (all) sufficiently small $\beta_i\in(0,1]$.
\end{definition}

Note that, if $D$ has only one component, then $(X,D)$ being (strongly) asymptotically log Fano just means that $(X,(1-\beta)D)$ is log Fano for sufficiently small $\beta$. Regarding the existence of KEE metrics on such pairs, Cheltsov--Rubinstein proposed the following conjecture.

\begin{conjecture}{\cite{CR15}}
\label{conjecture:CR-conj}
Let $(X,D)$ be a smooth asymptotically log Fano pair where $D$ is a smooth divisor.
Then $(X,(1-\beta)D)$ admits a KEE metric with sufficiently small cone angle $\beta$ along $D$ if and only if $(K_X+D)^{\dim X}=0$.
\end{conjecture}

One direction of this conjecture (the necessary part) has been verified by Cheltsov--Rubinstein \cite{CR16} (for dimension 2)
and Fujita \cite{F15} (for higher dimensions). Moreover, Cheltsov--Rubinstein \cite{CR15,CR16} confirmed the conjecture for all pairs in dimension 2 except one infinite family of pairs, and recently Cheltsov--Rubinstein--Zhang \cite{CRZ} confirmed the conjecture for all but 6 of these pairs. In this note we show that some of these remaining cases provide counterexamples to Conjecture \ref{conjecture:CR-conj} (see Section \ref{sec:dim 2}). In addition we provide other counterexamples in higher dimensions and  investigate the subtlety involved (see Section \ref{sec:high dim}).

\section{Examples and counterexamples in dimension 2}
\label{sec:dim 2}

\subsection{Preliminaries}
\label{sebsec:prelim-dim-2}
\hfill\\
\hfill\\
In this section, we let $(S,C)$ be a smooth asymptotically log del Pezzo pair.
Assume that $(K_S+C)^2=0$. 
For the sufficient part, in dimension 2, it is useful to
divide into two cases: when $K_S+C\sim 0$ and when $K_S+C\not\sim0$.
In the first case existence (and hence the conjecture \ref{conjecture:CR-conj}) follows from \cite[Corollary 1]{JMR} which resolved a conjecture of Donaldson \cite{Don}.
In the second case, Cheltsov--Rubinstein's classification of asymptotically
log del Pezzo surfaces \cite[Theorem 2.1]{CR15} reduces the task to $(\mathbb{F}_1,C)$ with $C\in|2Z_1+2F|$ or $(S_r,C_r)$ where $S_r$ the blow-up of $\mathbb{P}^1\times\mathbb{P}^1$ at $r$-points on a
bi-degree $(2,1)$ curve with no two on the same $(0,1)$ curve and $C_r$
its proper transform. In \cite{CR15} the first surface $(F_1,C)$ was
successfully treated using $\alpha$-invariant techniques and in a recent article \cite{CRZ} all but 6 of the pairs $(S_r,C_r)$ were handles using
$\delta$-invariant techniques (see \cite[Theorem 1.3]{CRZ}). Thus the conjecture seemed plausible, at least in dimension 2. Somewhat surprisingly, we show that nevertheless some of
$(S_1,C_1)$ and $(S_2,C_2)$ provide subtle counterexamples to Conjecture \ref{conjecture:CR-conj}.




More precisely, we will show that, for $r=1$ and $r=2$, some special configurations of $(S_r,C_r)$ are not uniformly K-stable. Here by `special' we mean that the blown up points on the $(1,2)$ curve are chosen in a specific way. 

To do this, we make use of the delta invariant defined by Fujita--Odaka \cite{FO} and we will show that, for some special $(S,C)$ from above, one has
$$\delta\big(S,(1-\beta)C\big)\leq1,\ \text{for sufficiently small $\beta$}.$$
This means that $(S,(1-\beta)C)$ is not uniformly K-stable by \cite{BJ}. Then some further argument will imply the non-existence of small cone angle KEE metrics (see Remark \ref{remark:r=1} and \ref{remark:r=2}).

To bound $\delta$-invariants from above, we use the following characterization of $\delta$-invariant (see \cite{FO,BJ}):
\begin{equation}
\label{equation:delta=A/S}
    \delta\big(S,(1-\beta)C\big)=\inf_Z\frac{A_{S,(1-\beta)C}(Z)}{S_{S,(1-\beta)C}(Z)}.
\end{equation}
$$$$
Here $Z$ runs through all the prime divisors over the surface $S$ and $A_{S,(1-\beta)C}(Z)$ denotes the log discrepancy of $Z$. For simplicity, we will write $A(Z):=A_{S,(1-\beta)C}(Z)$ in the following. Moreover, the quantity $S(Z):=S_{S,(1-\beta)C}(Z)$ is called the \emph{expected vanishing order} of $-K_S-(1-\beta)C$ along $Z$, which is defined by
$$S(Z):=\frac{1}{(-K_S-(1-\beta)C)^2}\int_0^{\tau(Z)}\vol(-K_S-(1-\beta)C-xZ)dx,$$
where $\tau(Z)$ denotes the pseudo-effective threshold of $-K_S-(1-\beta)C$ with respect to $Z$.
And as we will see, in some cases, the infimum in \eqref{equation:delta=A/S} is obtained by some specific $Z$ over $S$.

\subsection{Basic setup and notation}
\label{subsection:notaion}
\hfill\\
\hfill\\
In this subsection, we fix some notation, which will be used throughout this section.
Set
$$
\overline{S}:=\mathbb{P}^1\times\mathbb{P}^1, \qquad
\overline{C}:=\text{a smooth curve of bi-degree $(1,2)$}\subset\overline{S}.
$$
Denote by $\overline{F}$ a general vertical line of bi-degree $(1,0)$
 and by $\overline{H}$ a general horizontal line of bi-degree $(0,1)$.

let $([s:t],[u:v])$ be the bi-homogeneous coordinate system on $\overline{S}$. Then, up to a linear change of coordinates, we may assume that $\overline{C}$ is cut out by the equation $sv^2=tu^2$.


The linear system $|\overline{F}|$ contains exactly two curves
that are tangent to $\overline{C}$.
Denote them by
$
\overline{F}_0,\overline{F}_\infty,
$
and let
$$
\overline{p_0}:=\overline{F}_0\cap \overline{C},\qquad
\overline{p_\infty}:=\overline{F}_\infty\cap \overline{C}.\qquad
$$
In $([s:t],[u:v])$ coordinates, one simply has $\overline{F}_0=\{t=0\},\ \overline{F}_\infty=\{s=0\}$, $\overline p_0=([1:0],[1;0])$ and $\overline p_\infty=([0:1],[0:1])$. We also put $\overline{H}_0$ and $\overline{H}_\infty$ to be the horizontal $(0,1)$ curves that intersect $\overline{C}$ transversely at $\overline{p}_0$ and $\overline{p}_\infty$ respectively. So $\overline{H}_0=\{v=0\}$ and $\overline{H}_\infty=\{u=0\}$.

\begin{figure}[H]
\centering
\includegraphics[width=0.75\textwidth]{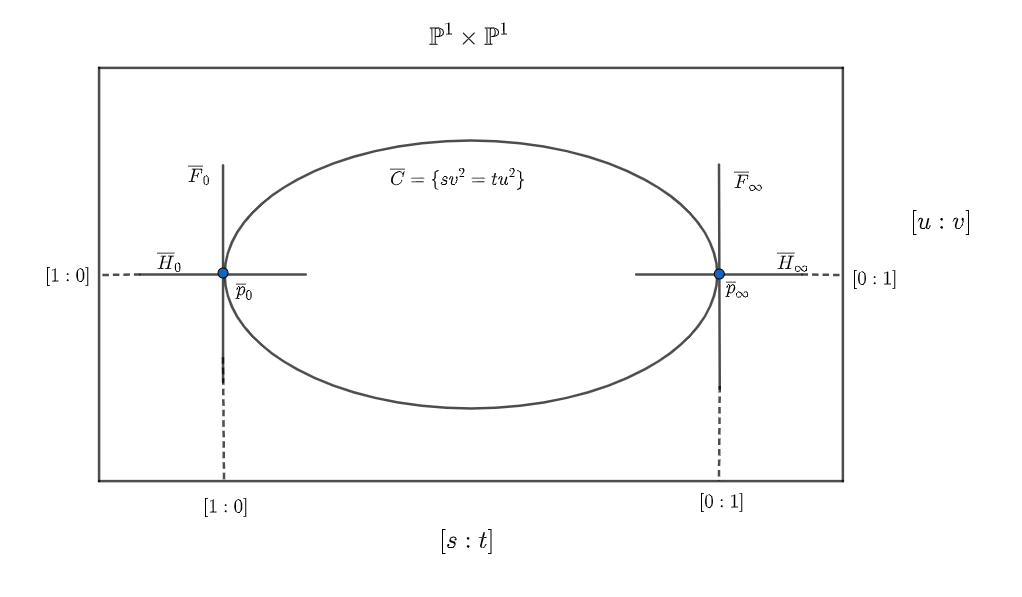}
\caption{\label{fig:S}}
\end{figure}

Choose some $r\in\mathbb{N}$ and let $\overline{F}_1,\ldots,\overline{F}_r$ be
distinct  bi-degree $(1,0)$ curves in $\overline{S}$
that are all different from the curves $\overline{F}_0$ and $\overline{F}_\infty$.
Then each intersection $\overline{F}_i\cap\overline{C}$ consists of two points.
For each $i=1,\ldots r$, let
$$
\overline{p}_i\in \overline{F}_i\cap \overline{C}
$$
be one of these two points.

Set
$$
I:=\{i_1,\ldots,i_r\}\subset \{0,1,...,r,\infty\},
$$
and let $S:=S_I$ denote the blow-up of $\overline{S}$ at the $r$ points $\{\overline p_i\}_{i\in I}\subset\{\overline p_0,\overline p_1,...,\overline p_r,\overline p_\infty\}$,
with $\pi \colon S \to \overline{S}$ being the blowup morphism.
Let us denote by
$$
E_{j}:=\pi^{-1}(\overline p_{j}), 
\qquad
j\in I, 
$$ 
the exceptional curves of $\pi$. To be precise,
we note that 
we are blowing-up $r$ of the $r+2$ points 
$\{\overline p_0,\overline p_1,...,\overline p_r,\overline p_\infty\}$.
Denote by
$$
F_0,F_1,\ldots,F_r,F_{\infty}
$$
the proper transform on the surface $S$ of the curves $\overline{F}_0,\overline{F}_1,\ldots,\overline{F}_r,\overline{F}_{\infty}$
(note that exactly $r$ of these are $-1$-curves and the remaining
two are $0$-curves). We also set $\overline{H}_i$ to be the horizontal $(0,1)$ curve passing through $\overline{p}_i$ and let $H_i$ be its proper transform on $S$.

Let $C$ be the proper transform of the curve $\overline{C}$,
so
$$C=\pi^*\overline{C}-\sum_{j\in I}E_j\sim\pi^*(\overline{F}+2\overline{H})-\sum_{j\in I}E_j.$$
For any sufficiently small rational number $\beta>0$, we put
$$\overline{L}_\beta:=-K_{\overline{S}}-(1-\beta)\overline{C}.$$
Then we have
$$\overline{L}_\beta\sim_\mathbb{Q}(1+\beta)\overline{F}+2\beta\overline{H}\sim_\mathbb{Q}\overline{F}+\beta\overline{C}.$$
Let
$$L_\beta:=-K_S-(1-\beta)C.$$
Then we have
$$L_\beta\sim_\mathbb{Q}\pi^*\overline{L}_\beta-\beta\sum_{j\in I}E_j\sim_\mathbb{Q}\pi^*\big((1+\beta)\overline{F}+2\beta\overline{H}\big)-\beta\sum_{j\in I}E_j\sim_\mathbb{Q}\pi^*\overline{F}+\beta C.$$
Note that $L_\beta$ is an ample $\mathbb{Q}$-line bundle for sufficiently small $\beta$, so that the pair $(S,C)$ is asymptotically log del Pezzo.

\subsection{Blowing up two special points}
\label{subsection:2-special}
\hfill\\
\hfill\\
In this part we set $r=2$ and $I=\{0,\infty\}$. So $(S,C)$ is obtained by blowing up $\overline{p}_0$ and $\overline{p}_\infty$. In this case, $L_\beta$ is ample for any $\beta\in(0,1]$. The main result is the following.

\begin{proposition}
\label{proposition:KEE-for-2-special}
$(S,(1-\beta)C)$ admits a KEE metric with cone angle $\beta$ along $C$ for $\beta\in(0,1]$
\end{proposition}

To show this, we use Tian's $\alpha_G$-invariant, where we take
$$
G:=\mathbb{C}^*\rtimes \mathbb{Z}_2.
$$

Note that $G\subset\operatorname{Aut}(\mathbb{P}^1)$. The action is simply given by multiplication and involution. If we embed $\mathbb{P}^1$ into $\overline{S}=\mathbb{P}^1\times\mathbb{P}^1$ as the  $(1,2)$ curve $\overline{C}$ (the map is given by $[x:y]\mapsto([x^2:y^2],[x:y])$), then the $G$-action extends to $(\overline{S},\overline{C})$. 
Namely, $G\subset\operatorname{Aut}(\overline{S},\overline{C})$. More specifically, for any $(\lambda,\iota)\in G$ and $([s:t],[u:v])\in\overline{S}$, the induced action is given by
$$
\lambda\cdot([s:t],[u:v])=([s:\lambda^2t],[u:\lambda v]),\ \iota\cdot([s:t],[u:v])=([t:s],[v:u]).
$$
This $G$-action lifts to $(S,C)$ since we are blowing up $\overline{p}_0$ and $\overline{p}_\infty$. In particular, the curves $F_0$, $E_0$, $H_0$, $F_\infty$, $E_\infty$, $H_\infty$ and $C$ are all $\mathbb{C}^*$-invariant and 
$$
\iota(F_0)=F_\infty,\ \iota(E_0)=E_\infty,\ \iota(H_0)=H_\infty,\ \iota(C)=C.
$$

\begin{figure}[H]
\centering
\includegraphics[width=0.7\textwidth]{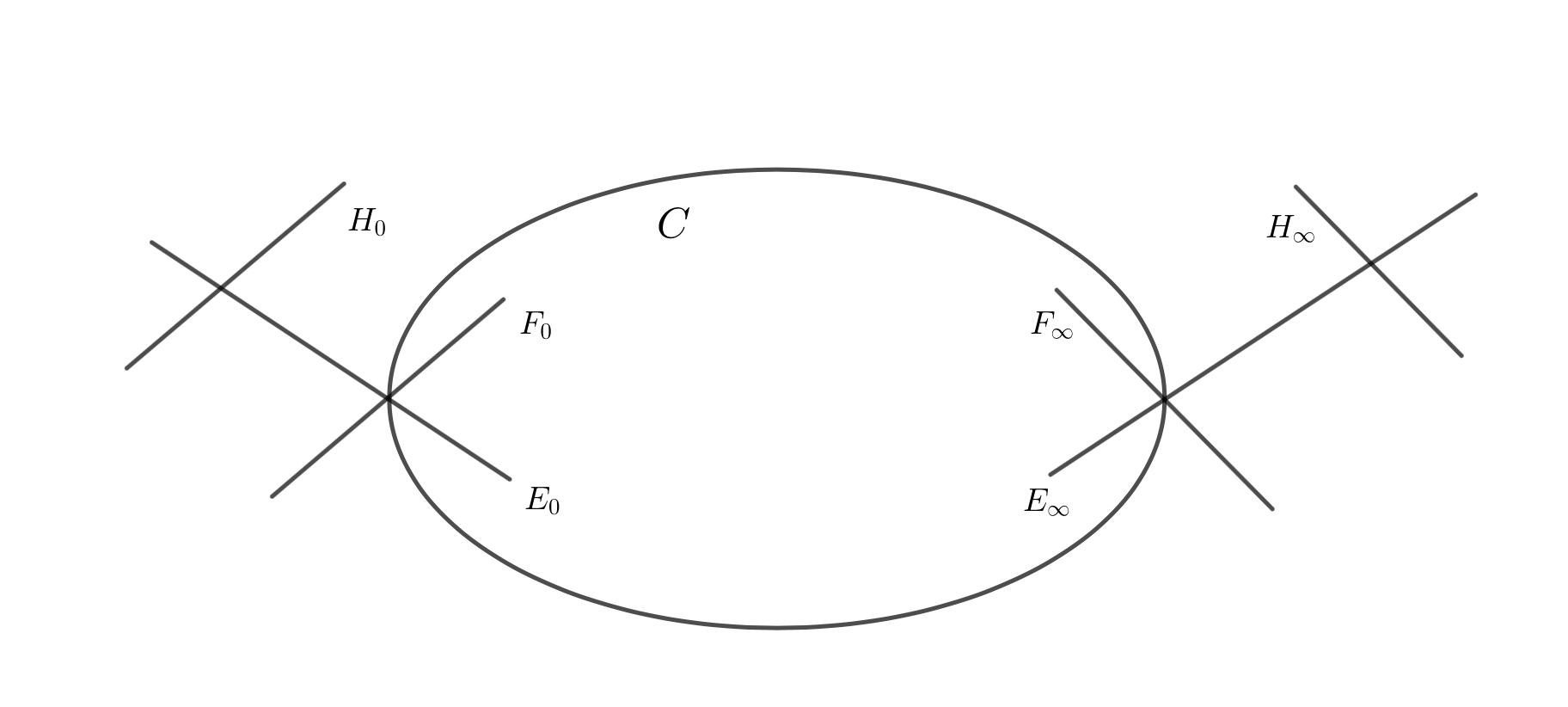}
\caption{\label{fig:2b}}
\end{figure}


\begin{remark}
Due to the existence of $G$-action, Proposition \ref{proposition:KEE-for-2-special} implies that the log Fano pair $\big(S,(1-\beta)C\big)$ is log K-polystable. Moreover, when $\beta=1$, this recovers the well-known existence of KE metrics on $Bl_3\mathbb{P}^2$.
\end{remark}

To prove Proposition \ref{proposition:KEE-for-2-special}, it is enough (see \cite[Theorem 2,~Lemma 6.11]{JMR}) to show the following 

\begin{proposition}
\label{proposition:alpha_G}
One has $\alpha_G\big(S,(1-\beta)C\big)=1$ for $\beta\in(0,1]$.
\end{proposition}
Here $\alpha_G\big(S,(1-\beta)C\big)$ is defined as
$$
\alpha_G\big(S,(1-\beta)C\big):=\mathrm{sup}\left\{\lambda\in\mathbb{Q}\ \Bigg|\;%
\aligned
&\text{the log pair}\ \left(S,(1-\beta)C+\lambda \mathcal{D}\right)\ \text{is log canonical}\\
&\text{for every G-invariant $\mathbb{Q}$-linear system}\ \mathcal{D}\sim_{\mathbb{Q}} L_\beta
\endaligned\Bigg.\right\}.
$$

\begin{remark}
In fact, to compute our $\alpha_G$, it suffices to consider $G$-invariant divisors. Indeed, as $\mathbb{C}^*$ is abelian, so any $\mathbb{C}^*$-invariant linear system must contain a $\mathbb{C}^*$-divisor $D$. Then it is enough to look at $\frac{1}{2}(D+\iota(D))$, as in \cite[Section 7]{CS18}.
\end{remark}

\begin{proof}[Proof of Proposition \ref{proposition:alpha_G}]

We will show that for any $G$-invariant divisor $D\sim_\mathbb{Q} L_\beta$ the pair $(S,(1-\beta)C+D)$ is log canonical,
but $(S,(1-\beta)C+\lambda D)$ is not for $\lambda > 1$.  The the Picard group of $S$ has basis $[H]:=[\pi^*\overline{H}_0]$, $[F]:=[\pi^*\overline{F}_0]$, $[E_0]$ and $[E_\infty]$.  In this basis we have $[F_0]=[F]-[E_0]$, $[F_\infty]=[F]-[E_\infty]$, $[H_0]=[H]-[E_0]$ and $[H_\infty]=[H]-[E_\infty]$.
An anti-canonical divisor is given by \begin{equation}-K_S=2H+2F-E_0-E_\infty\label{eq:KS}.\end{equation} Note, that the $\mathbb{C}^*$-invariant curves on $S$ are given by $E_0,E_\infty, F_0, F_\infty, H_0, H_\infty$ and the strict transforms $C_{[\alpha:\beta]}$ of the curves $\overline{C}_{[\alpha:\beta]}=[\alpha sv^2= \beta tu^2]\subset \overline{S}$ for $[\alpha:\beta]\in \mathbb{P}^1 \setminus\{0,\infty\}$. These curves are all linearly equivalent to $2H+F-E_0-E_\infty$ and
each of them passes through the intersection points $p_0 = E_0 \cap F_0$ and $p_\infty = E_\infty \cap F_\infty$ intersecting all four curves transversely. The curves $C_{[\alpha:\beta]}$ also intersect each other pairwise transversely in these two points.

Being $G$-invariant the divisor $(1-\beta)C+D \sim -K_S$ has to have the form
\[e(E_0+E_\infty)+f(F_0+F_\infty)+h(H_0+H_\infty)+ \sum \gamma_i C_{[\alpha_i:\beta_i]}.\]
We set $\gamma := \sum_i \gamma_i$. Passing to the classes in the Picard group of $S$ we get
\[[(1-\beta)C+D]=(e-f-h-\gamma)([E_0]+[E_\infty])+(2f+\gamma)\cdot [F]+(2h+2\gamma) \cdot [H].\]
Comparing coefficient with $-K_S$ in (\ref{eq:KS}) gives $h=1-\gamma$, $f=1-\gamma/2$ and \(e=f+h+\gamma-1 = 1- \gamma/2\). 
Therefore all coefficients of $(1-\beta)C+D$ are less or equal to $1$. A log resolution of $\big(S,(1-\beta)C+D\big)$ is given by further blowing up $S$ in
$p_0$ and $p_\infty$. The multiplicity of $(1-\beta)C+D$ in these points is $\gamma+e+f=2$. This implies that
the discrepancy at the corresponding exceptional divisor is $-1$ and rescaling $D$ by a constant $\lambda>1$ will further decrease this discrepancy.
\end{proof}

\subsection{Counterexamples}
\label{subsection:S(Z)}
\hfill\\
\hfill\\
In this part we carry out some explicit computation and give upper bounds for $\delta\big(S,(1-\beta)C\big)$. This will give us some counterexamples to Conjecture \ref{conjecture:CR-conj} (see Remark \ref{remark:r=1} and \ref{remark:r=2}).

Recall that, for any prime divisor $Z$ over the surface $S$, we have the expected vanishing order
$$S(Z):=\frac{1}{(L_\beta)^2}\int_0^{\tau(Z)}\vol(L_\beta-xZ)dx.$$

\begin{proposition}
\label{propsition:Z-on-S}
For any $r\ge1$, one has
\begin{equation*}
S(Z)=
    \begin{cases}
    \frac{1}{2}+\frac{6-r}{8}\beta+O(\beta^2),\ Z=E_i,\ i\in I;\\
    \frac{1}{2}+\frac{6-r}{8}\beta+O(\beta^2),\ Z=F_i,\ i\in I;\\
    \frac{\beta}{2}+\frac{r-4}{24}\beta^2+O(\beta^3),\ Z=C;\\
    \frac{1}{2}+\frac{4-r}{8}\beta+O(\beta^2),\ Z\in|\pi^*\overline{F}|.
    \end{cases}
\end{equation*}
\end{proposition}
This follows from explicit computation; see \cite{CRZ} for more details. Note that, in the above proposition, the prime divisor $Z$ is on the surface $S$. We can also consider $Z$ \emph{over} $S$. 

\begin{proposition}
\label{proposition:Z-over-S}
We have the following
\begin{enumerate}
    \item Suppose that $p\in C$ is away from any $F_i$ or $E_i$, where $i\in I$. Let $\widetilde{S}\xrightarrow{\sigma} S$ be the blow-up of $p$ and let $Z$ be the exceptional curve of $\sigma$. Then
    $$S(Z)=\frac{1}{2}+\frac{12-r}{8}\beta+O(\beta^2),\ r\geq0$$
    \item Suppose that $0\in I$, i.e. $\overline{p}_0$ is blown up. Put $p_0=E_0\cap C$. Let $\widetilde{S}\xrightarrow{\sigma} S$ be the blow-up of $p_0$ and let $Z$ be the exceptional curve of $\sigma$. Then
    $$S(Z)=1+\frac{6-r}{4}\beta+O(\beta^2),\ r\geq1.$$
    Moreover, for $r=2$, we have exactly
    $$S(Z)=1+\beta;$$
    \item Suppose that $p=E_i\cap C$ or $p=F_i\cap C$ for some $i\in I$ and $i\neq0$ or $\infty$. Let $\widetilde{S}\xrightarrow{\sigma} S$ be the blow-up of $p$ and let $Z$ be the exceptional curve of $\sigma$. Then
    $$S(Z)=\frac{1}{2}+\frac{14-r}{8}\beta+O(\beta^2),\ r\geq1;$$
    \item Suppose $0\notin I$. Let $p_0=F_0\cap C$. Let $\widetilde{S}\xrightarrow{\sigma} S$ be the blow-up of $p_0$ and let $G$ be the exceptional curve of $\sigma$. Let $\widetilde{C}$ be the proper transform of $C$ on $\widetilde{S}$. Put $q_0=G\cap\widetilde{C}$. Let $\widehat{S}\xrightarrow{\tau} \widetilde{S}$ be the blow-up of $q_0$ and let $Z$ be the exceptional curve of $\tau$.
    Then we have
    \begin{equation*}
    S(Z)=
        \begin{cases}
        1+2\beta,\ r=0,\\
        1+\frac{8-r}{4}\beta+O(\beta^2),\ r\geq1.\\
        \end{cases}
    \end{equation*}
\end{enumerate}
\end{proposition}

Again, this follows from elementary computation. For the reader's convenience, we include the proof of case (2) with $r=2$. The computation for other cases is similar.

\begin{proof}[Proof of Proposition \ref{proposition:Z-over-S}(2) with $r=2$]
In this case, $S$ is obtained by blowing up $\overline{p}_0$ and another point $\overline{p}_i$ (possibly $\overline{p}_\infty$) on $\overline{C}$. Then we have
$$L_\beta\sim_\mathbb{Q}\pi^*\big((1+\beta)\overline{F_0}+\beta\overline{H}_0+\beta\overline{H}_i\big)-\beta E_0-\beta E_i=(1+\beta)(F_0+E_0)+\beta(H_0+H_i).$$
Now let $\widetilde{S}\xrightarrow{\sigma}S$ be the blow-up of $p_0=E_0\cap C$ with $Z$ being the exceptional curve of $\sigma$. Let $\widetilde{F}_0$, $\widetilde{E}_0$, $\widetilde{H}_0$ and $\widetilde{H}_i$ be the proper transforms of $F_0$, $E_0$, $H_0$ and $H_i$ on $\widetilde{S}$ respectively. Then we have
$$\sigma^*L_\beta-xZ\sim_\mathbb{Q}(1+\beta)(\widetilde{F}_0+\widetilde{E}_0)+\beta(\widetilde{H}_0+\widetilde{H}_i)+(2+2\beta-x)Z.$$
Note that
\begin{equation*}
    \begin{cases}
    (\sigma^*L_\beta-xZ)\cdot\widetilde{F}_0=(\sigma^*L_\beta-xZ)\cdot\widetilde{E}_0=\beta-x,\\
    (\sigma^*L_\beta-xZ)\cdot\widetilde{H}_0=(\sigma^*L_\beta-xZ)\cdot\widetilde{H}_i=1,\\
    (\sigma^*L_\beta-xZ)\cdot Z=x.\\
    \end{cases}
\end{equation*}
So $\sigma^*L_\beta-xZ$ is nef for $x\in[0,\beta]$. Thus we have
$$\vol\big(\sigma^*L_\beta-xZ\big)=\big(\sigma^*L_\beta-xZ\big)^2=4\beta+2\beta^2-x^2,\ x\in[0,\beta].$$
And for $x\geq\beta$, \cite[Corollary 2.8]{CZ} implies
$$\vol\big(\sigma^*L_\beta-xZ\big)=\vol\big(\sigma^*L_\beta-xZ-\frac{x-\beta}{2}(\widetilde{F}_0+\widetilde{E}_0)\big).$$
Note that
\begin{equation*}
    \begin{cases}
    \big(\sigma^*L_\beta-xZ-\frac{x-\beta}{2}(\widetilde{F}_0+\widetilde{E}_0)\big)\cdot\widetilde{F}_0=\big(\sigma^*L_\beta-xZ-\frac{x-\beta}{2}(\widetilde{F}_0+\widetilde{E}_0)\big)\cdot\widetilde{E}_0=0,\\
    \big(\sigma^*L_\beta-xZ-\frac{x-\beta}{2}(\widetilde{F}_0+\widetilde{E}_0)\big)\cdot\widetilde{H}_0=\big(\sigma^*L_\beta-xZ-\frac{x-\beta}{2}(\widetilde{F}_0+\widetilde{E}_0)\big)\cdot\widetilde{H}_i=1-\frac{x-\beta}{2},\\
    \big(\sigma^*L_\beta-xZ-\frac{x-\beta}{2}(\widetilde{F}_0+\widetilde{E}_0)\big)\cdot Z=\beta.\\
    \end{cases}
\end{equation*}
So $\sigma^*L_\beta-xZ-\frac{x-\beta}{2}(\widetilde{F}_0+\widetilde{E}_0)$ is nef for $x\in[\beta,2+\beta]$. Thus for $x\in[\beta,2+\beta]$ we have
$$\vol\big(\sigma^*L_\beta-xZ\big)=\big(\sigma^*L_\beta-xZ-\frac{x-\beta}{2}(\widetilde{F}_0+\widetilde{E}_0)\big)^2=4\beta+3\beta^2-2\beta x.$$
Now for $x\geq 2+\beta,$ we use Zariski decomposition \cite[Corollary 2.7]{CZ}. Solve the following equations:
\begin{equation*}
    \begin{cases}
    \big(a\widetilde{F}_0+b\widetilde{E}_0+c\widetilde{H}_0+d\widetilde{H}_i+(2+2\beta-x)Z\big)\cdot \widetilde{F}_0=0\\
    \big(a\widetilde{F}_0+b\widetilde{E}_0+c\widetilde{H}_0+d\widetilde{H}_i+(2+2\beta-x)Z\big)\cdot \widetilde{E}_0=0\\
    \big(a\widetilde{F}_0+b\widetilde{E}_0+c\widetilde{H}_0+d\widetilde{H}_i+(2+2\beta-x)Z\big)\cdot \widetilde{H}_0=0\\
    \big(a\widetilde{F}_0+b\widetilde{E}_0+c\widetilde{H}_0+d\widetilde{H}_i+(2+2\beta-x)Z\big)\cdot \widetilde{H}_i=0\\
    \end{cases}
\end{equation*}
We get
$$a=b=c=d=2+2\beta-x.$$
This implies, for $x\in[2+\beta,2+2\beta]$, we have
    \begin{align*}
        \vol\big(\sigma^*L_\beta-xZ\big)&=\vol\big((2+2\beta-x)(\widetilde{F}_0+\widetilde{E}_0+\widetilde{H}_0+\widetilde{H}_i+Z)\big)\\
         &=(2+2\beta-x)^2\big(\widetilde{F}_0+\widetilde{E}_0+\widetilde{H}_0+\widetilde{H}_i+Z\big)^2\\
         &=(2+2\beta-x)^2.\\
    \end{align*}
So we can compute
    \begin{align*}
        \int_0^{\tau(Z)}\vol(\sigma^*L_\beta-xZ\big)dx&=\int_0^{2+2\beta}\vol(\sigma^*L_\beta-xZ\big)dx\\
        &=\int_0^\beta(4\beta+2\beta^2-x^2)dx+\int_\beta^{2+\beta}(4\beta+3\beta^2-2\beta x)dx + \\
        &\quad +\int_{2+\beta}^{2+2\beta}(2+2\beta-x)^2dx\\
        &=4\beta+6\beta^2+2\beta^3.\\
    \end{align*}
Thus we get
$$S(Z)=\frac{1}{L^2_\beta}\int_0^{\tau(Z)}\vol(\sigma^*L_\beta-xZ\big)dx=\frac{4\beta+6\beta^2+2\beta^3}{4\beta+2\beta^2}=1+\beta.$$

\end{proof}

Note that Proposition \ref{proposition:Z-over-S} has the following consequence.
\begin{corollary}
\label{corollary:delta-upper-bound}
We have the following upper bound for $\delta$-invariant.
\begin{enumerate}
    \item If $0\in I$, then
    $$\delta\big(S,(1-\beta)C\big)\leq1+\frac{r-2}{4}\beta+O(\beta^2),\ r\geq1.$$
    Moreover, when $r=2$, we have exactly
    $$\delta\big(S,(1-\beta)C\big)=1,$$
    and the infimum of \eqref{equation:delta=A/S} is obtained by the $Z$ in Proposition \ref{propsition:Z-on-S}(2).
    \item If $0\notin I$, then
    $$\delta\big(S,(1-\beta)C\big)\leq1+r\beta+O(\beta^2),\ r\geq0.$$
    Moreover, when $r=0$, we have exactly
    $$\delta\big(S,(1-\beta)C\big)=1,$$
    and the infimum of \eqref{equation:delta=A/S} is obtained by the $Z$ in Proposition \ref{propsition:Z-on-S}(4).
\end{enumerate}
\end{corollary}

\begin{proof}
(1) Let $Z$ be the divisor in Proposition \ref{proposition:Z-over-S}(2). Then we have
$$A(Z)=2-(1-\beta)=1+\beta.$$
Using \eqref{equation:delta=A/S}, for $r\geq1$ we get
$$\delta\big(S,(1-\beta)C\big)\leq\frac{A(Z)}{S(Z)}=\frac{1+\beta}{1+\frac{6-r}{4}\beta+O(\beta^2)}=1+\frac{r-2}{4}\beta+O(\beta^2).$$
When $r=2$, we have
$$\delta\big(S,(1-\beta)C\big)\leq\frac{A(Z)}{S(Z)}=\frac{1+\beta}{1+\beta}=1.$$

To see this is actually an equality, we need to use some deeper results. Suppose that $\infty\in I$, so $(S,C)$ is obtained by blowing up $\overline{p}_0$ and $\overline{p}_\infty$. Then Proposition \ref{proposition:KEE-for-2-special} implies that $(S,(1-\beta)C)$ is K-polystable, so we have the other direction:
$$\delta\big(S,(1-\beta)C\big)\geq1.$$

If $\infty\notin I$, then the $\mathbb{C}^*$-action on $(\overline{S},\overline{C})$ (see Subsection \ref{subsection:2-special} for details about this action) induces a K-polystable degeneration of the log Fano pair $(S,(1-\beta)C)$. To be more precise, we are blowing up $\overline{p}_0$ and another point $\overline{p}_i$ on the $(1,2)$ curve $\overline{C}$. Then $\mathbb{C}^*$-action on $(\overline{S},\overline{C})$ fixes $\overline{p}_0$ but moves $\overline{p}_i$ towards $\overline{p}_\infty$. So this action induces a degeneration from $(S,C)$ towards the above K-polystable pair obtained by blowing up $\overline{p}_0$ and $\overline{p}_\infty$. By the lower semi-continuity of $\delta$-invariant (cf. \cite{BL}), we again obtain
$$\delta\big(S,(1-\beta)C\big)\geq1.$$

(2) Let $Z$ be the divisor in Proposition \ref{proposition:Z-over-S}(4). Then we have
$$A(Z)=3-2(1-\beta)=1+2\beta.$$
Using \eqref{equation:delta=A/S}, for $r\geq0$ we get
$$\delta\big(S,(1-\beta)C\big)\leq\frac{A(Z)}{S(Z)}=\frac{1+2\beta}{1+\frac{8-r}{4}\beta+O(\beta^2)}=1+r\beta+O(\beta^2).$$
When $r=0$, we have
$$\delta\big(S,(1-\beta)C\big)\leq\frac{A(Z)}{S(Z)}=\frac{1+2\beta}{1+2\beta}=1.$$
To see this is actually an equality, we use \cite[Proposition 7.4]{CR15}, which implies that $(S,(1-\beta)C)$ is K-polystable, so we have the other direction:
$$\delta\big(S,(1-\beta)C\big)\geq1.$$
\end{proof}

\begin{remark}
\label{remark:r=1}
Suppose that $I=\{0\}$, then Corollary \ref{corollary:delta-upper-bound}(1) implies that
$$\delta\big(S,(1-\beta)C\big)<1$$
for sufficiently small $\beta$. This means that $(S,(1-\beta)C)$ does not admit a KEE metric with sufficiently small cone angle $\beta$. So Conjecture \ref{conjecture:CR-conj} fails in this case. 
Note that there is another way to obtain $\delta\big(S,(1-\beta)C\big)<1$, which relies on the toric calculation in \cite[Section 7]{BJ}.
Indeed, if $I=\{0\}$, then $S\cong Bl_2\mathbb{P}^2$ is a toric surface. One can determine the polytope $P_{\beta}$ of $L_{\beta}$ and its barycenter $bc_\beta$. Let $Z$ be the divisor in Proposition \ref{proposition:Z-over-S}(2). Then by \cite[Corollary 7.7]{BJ}, $S(Z)$ can be explicitly computed as $Z$ gives rise to a toric valuation $v_Z$. More specifically, following the notation therein, we have 
$$S(Z)=\langle bc_\beta,v_Z \rangle-\psi(v_Z)=\frac{4(1+\beta)^2}{4+3\beta},$$
where $bc_\beta=\left(\frac{-(4\beta^2+9\beta+6)}{3(4+3\beta)},\frac{-(7\beta^2+12\beta)}{3(4+3\beta)}\right)$, $v_Z=(2,1)$ and $\psi(v_Z)=-(2+3\beta)$. So that
$$\delta\big(S,(1-\beta)C\big)\leq\frac{A(Z)}{S(Z)}=\frac{4+3\beta}{4+4\beta}<1.$$

\begin{figure}[H]
\centering
\includegraphics[width=0.8\textwidth]{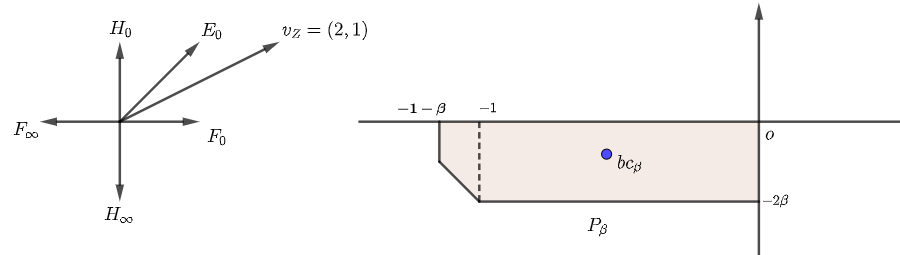}
\caption{\label{fig:toric}The fan of $S$ and the polytope $P_\beta$ of $L_\beta$}
\end{figure}

\end{remark}


\begin{remark}
\label{remark:r=2}
Let us take a closer look at the case when $I=\{0,i\}$ with $i\neq\infty$. Then $\operatorname{Aut}(S,C)$ is discrete. We claim that $(S,(1-\beta)C)$ does not admit a KEE metric with sufficiently small cone angle. If this is not the case, then the existence of a KEE metric implies the properness of K-energy, and hence the pair $(S,(1-\beta)C)$ is uniformly K-stable. So we should have $\delta\big(S,(1-\beta)C\big)>1$, contradicting Corollary \ref{corollary:delta-upper-bound}(1). So Conjecture \ref{conjecture:CR-conj} fails in this case as well. There is another way to see this.
Indeed, as we have seen in the proof of Corollary \ref{corollary:delta-upper-bound}(1), $(S,(1-\beta)C)$ admits a K-polystable degeneration, which implies that $(S,(1-\beta)C)$ cannot be K-stable. So $(S,(1-\beta)C)$ is strictly K-semistable and it cannot admit a KEE metric. In other words, we get a family of strictly K-semistable log Fano surfaces degenerating to a K-polystable log Fano surface. This can be thought of as a 2-dimensional log version of Tian's Mukai-Umemura example (see \cite{T97}).
\end{remark}

In the following table we summarize what is known about the K-stability of the asymptotically log del Pezzo surfaces $(S,C)=(S_I,C)$.
\begin{center}
\begin{table}[h]
\begin{tabular}{ll}
\toprule
 I & K-stability\\
 \toprule
 $\{i_1\} \subset \{0,\infty\}$ & K-unstable\\
 $\{i_1\} \not\subset \{0,\infty\}$ & ? \\
 \midrule
 $\{i_1,i_2\} = \{0,\infty\}$ & strictly K-polystable\\
 $\emptyset \neq \{i_1,i_2\} \cap \{0,\infty\} \neq \{0,\infty\}$ & strictly K-semistable\\
 $\{i_1,i_2\} \cap \{0,\infty\} = \emptyset$ & ?\\
 \midrule
 $3 \leq \#I \leq 6$ & ? \\
 \midrule
 $\#I \geq 7$ & K-stable\\
 \bottomrule\\
\end{tabular}
\caption{K-stability for $(S_I,C)$}
\label{tab:kstab}
\end{table}
\end{center}
\begin{remark}
In the cases where we know the answer according to Table~\ref{tab:kstab}, it turns out that K-(semi/poly-)stability coincides with the GIT-(semi/poly-)stability of the point configuration on $\overline{C} \cong \mathbb{P}^1$ consisting of the blowup centers and the two special points $\overline{p}_0$ and $\overline{p}_\infty$. More precisely, we consider $(\overline{p}_0,\overline{p}_\infty,(\overline{p}_{i})_{i \in I}) \in (\mathbb{P}^1)^{\#I+2}$ and ask for the stability in the GIT sense of this point with respect to the diagonal $\operatorname{SL(2)}$-action on $(\mathbb{P}^1)^{\#I+2}$ and the unique $\operatorname{SL(2)}$-linearization of $\mathcal{O}(1,\ldots,1)$. In the light of this observation it is natural to expect that the remaining cases should all be K-stable, as the corresponding point configurations are indeed GIT-stable.
\end{remark}

\section{Higher dimensional counterexamples and further discussion}
\label{sec:high dim}
In this section we investigate the Cheltsov--Rubinstein program in higher dimensions.
\subsection{Product spaces}
\hfill\\
\hfill\\
There are also simple counterexamples to Conjecture \ref{conjecture:CR-conj} by taking products of log Fano pairs. Let $X_1$ and $X_2$ be two smooth Fano varieties. Suppose that $F\in|-K_X|$ is a smooth divisor. Put
$$X:=X_1\times X_2,\ D:=F\times X_2.$$
Then in particular, $-K_X=p_1^*(-K_{X_1})+p_1^*(-K_{X_2})$ and $D\in|-p_1^*K_{X_1}|$, where $p_1$ and $p_2$ are the natural projections from $X$ to $X_1$ and $X_2$ respectively. It is clear that
$(X,D)$ is an asymptotically log Fano pair, as for any $\beta\in(0,1]$,
$$-K_X-(1-\beta)D\sim_\mathbb{Q}\beta\cdot p_1^*(-K_{X_1})+p_2^*(-K_{X_2})$$
is ample. Moreover, $-(K_X+D)=p_2^*(-K_{X_2})$ is a nef divisor with $(K_X+D)^n=0$.

On the other hand, from the definition of $\delta$-invariant, we clearly have
\begin{equation}
\label{equ:prod-delta}
    \delta\big(X,(1-\beta)D\big)\leq\min\{\delta\big(X_1,(1-\beta)F\big),\delta(X_2)\}.
\end{equation}
So in particular, if $X_2$ is a K-unstable Fano manifold with $\delta(X_2)<1$, then $\delta\big(X,(1-\beta)D\big)<1$ as well. So in this case the pair $\big(X,(1-\beta)D\big)$ cannot admit any KEE metric.

\begin{example}
Take $X_1=\mathbb{P}^2$ and $X_2=\mathrm{Bl}_p\mathbb{P}^2$. Let $F$ be a smooth cubic curve on $X_1$. Then the pair $(X,D)$ we constructed above is asymptotically log Fano with $(K_X+D)^4=0$. And the log pair $\big(X,(1-\beta)D\big)$ does not admit any KEE metric for $\beta\in(0,1]$. So Conjecture \ref{conjecture:CR-conj} fails in this case. 
\end{example}

\begin{remark}
\eqref{equ:prod-delta} is actually an equality by the recent work \cite{Z19}.
\end{remark}

\subsection{K-stability of the base}
\hfill\\
\hfill\\
By Shokurov's base-point-free theorem, it is easy to see that if $(X,D)$ is asymptotically log Fano, then the divisor $-(K_X+D)$ is semi-ample and we let $\phi:X\to Y$ be its corresponding ample model (i.e. $\phi$ has connected fibers and $-(K_X+D)=\phi^*L$ for some ample divisor $L$ on $Y$). Since $-(K_X+D)$ is not big by assumption, we have $\dim X>\dim Y$ and in particular, $\phi$ is not birational. As $K_X+D\sim_{\phi.\bQ} 0$ and $(X,D)$ is lc, we can write $K_X+D\sim_\bQ \phi^*(K_Y+B+M)$ for some effective divisor $B$ (the boundary part) and some pseudo-effective divisor $M$ (the moduli part) by the canonical bundle formula \cite[Theorem 8.5.1]{Kol07}. The example of product varieties above suggests that in order for Conjecture \ref{conjecture:CR-conj} to be true, we may need to impose some conditions on the K-stability of the generalized pair $(Y,B+M)$. Here we give a definition of uniform K-stability and K-semistability of a generalized klt log Fano pair similar to the valuative criterion of Fujita \cite{Fuj19} and Li \cite{Li17}.  

\begin{definition}
Let $(Y, B+M)$ be a projective generalized klt pair such that $-(K_Y+B+M)$ is ample. 
\begin{enumerate}
    \item For any prime divisor $E$ over $Y$, We define 
    \[
    S_{Y,B+M}(E):=\frac{1}{\vol(-K_Y-B-M)}\int_{0}^\infty\vol(-K_Y-B-M-tE)dt.
    \]
    \item We say that $(Y,B+M)$ is \emph{K-semistable} if for any prime divisor $E$ over $Y$, we have 
    $A_{Y,B+M}(E)\geq S_{Y,B+M}(E)$.
    \item We say that $(Y,B+M)$ is \emph{uniformly K-stable} if there exists $\epsilon>0$ such that for any prime divisor $E$ over $Y$, we have
    $A_{Y,B+M}(E)\geq (1+\epsilon)S_{Y,B+M}(E)$.
\end{enumerate}
\end{definition}

In the following proposition we show that K-semistability of the base $(Y,B+M)$ is necessary for Conjecture \ref{conjecture:CR-conj} to hold for $(X,D)$. 

\begin{proposition}
Notation as above. Assume that $(X,(1-\beta)D)$ admits KEE metric for all sufficiently small cone angle $\beta>0$. Then $(Y,B+M)$ is K-semistable.
\end{proposition}

\begin{proof}
Let $E$ be a prime divisor over $Y$.
Let $\pi_Y :Y'\to Y$ be a proper birational morphism that extracts $E$ as a Cartier divisor. Let $X'$ be the normalization of the main component of $X\times_Y Y'$ with projections $\pi_X :X'\to X$ and $\phi':X'\to Y'$. Let $L'=\pi_Y^*L$ and $D'=\pi_X^*D$. 
Then or any ample line bundle $L_X$ on $X$, we set
\[
\vol(L_X-tE):=\vol(\pi_X^* L_X- t\phi'^{*}(E))
\]
where $n=\dim X$. We define the expected vanishing order $S_{X,(1-\beta)D}(E)$ of the log Fano pair $(X,(1-\beta)D)$ along $E$ by
\[
S_{X,(1-\beta)D}(E)=\frac{1}{\vol(-K_X-(1-\beta)D)}\int_0^\infty \vol(-K_X-(1-\beta)D)-tE) \rd t.
\]
We claim that
\begin{equation} \label{eq:limS}
    \liminf_{\beta\to 0+} S_{X,(1-\beta)D}(E) \ge S_{Y,B+M}(E).
\end{equation}
Let $r=\dim X - \dim Y\ge 1$, let $C=\binom{n}{r}$ and let $F$ be a general fiber of $\phi$. As $(X,D)$ is asymptotically log Fano, $\beta D\sim_{\phi.\bQ} -(K_X+(1-\beta)D)=\phi^*L+\beta D$ is $\phi$-ample for some $0<\beta\ll 1$ and we have 
\[
\vol(-K_X-(1-\beta)D)=\left( (-K_X-(1-\beta)D)^n \right) = C\beta^r ((\phi^*L)^{n-r}\cdot D^r) + O(\beta^{r+1})
\]
since $\phi$ has relative dimension $r$. As $L$ is ample and $D$ is $\phi$-ample, it is easy to see that $((\phi^*L)^{n-r}\cdot D^r)=(L^{n-r}) (D^r\cdot F)>0$, hence
\begin{equation} \label{eq:vol(L+beta D)}
    \vol(-K_X-(1-\beta)D)=C\beta^r (L^{n-r}) (D^r\cdot F) + O(\beta^{r+1}).
\end{equation}

For any $t\ge 0$ such that $\vol(L-tE)>0$ and any $\epsilon>0$, by Fujita's approximation theorem (see e.g. \cite[Theorem D]{LM09}) we may assume that (after possibly replacing $\pi_Y$ by another birational morphism) there exists $\bQ$-divisors $A$ and $N$ on $Y'$ such that $A$ is ample, $N$ is effective, $L'-tE=A+N$ and $\vol(A)=(A^{n-r}) > \vol(L-tE)-\epsilon$. As $D$ is $\phi$-ample, $D'$ is $\phi'$-ample, thus $\phi'^*A+\beta D'$ is ample for sufficiently small $\beta>0$. It follows that
\begin{align*}
    \vol(-K_X-(1-\beta)D-tE) & \ge \vol(\phi'^*(L'-tE)+\beta D') \\
     & \ge \vol(\phi'^*A+\beta D') \\
     & = \left( (\phi'^*A+\beta D')^n \right) = C\beta^r (A^{n-r}) (D^r\cdot F) + O(\beta^2)
\end{align*}
where the last equality follows from the projection formula and the ampleness (resp. $\phi$-ampleness) of $A$ (resp. $D')$ as before. In particular, we have
\begin{align*}
\liminf_{\beta\to 0+} \beta^{-r}\vol(-K_X-(1-\beta)D-tE) & \ge C\cdot (A^{n-r}) (D^r\cdot F) \\
 & > C \cdot \left( \vol(L-tE)-\epsilon \right) (D^r\cdot F).
\end{align*}
As this holds for all $\epsilon>0$, we obtain
\[
\liminf_{\beta\to 0+} \beta^{-1}\vol(-K_X-(1-\beta)D-tE) \ge C \cdot \vol(L-tE)\cdot (D^r\cdot F).
\]
Therefore by Fatou's lemma we see that
\begin{equation} \label{eq:lim-integral}
    \liminf_{\beta\to 0+} \frac{1}{\beta^r} \int_0^\infty \vol(-K_X-(1-\beta)D-tE) \rd t \ge C \cdot (D^r\cdot F) \int_0^\infty \vol(L-tE) \rd t.
\end{equation}
The claimed inequality \eqref{eq:limS} then follows by combining \eqref{eq:vol(L+beta D)} and \eqref{eq:lim-integral} as $L\sim_\bQ -(K_Y+B+M)$.

We now proceed to show that $(Y,B+M)$ is K-semistable, i.e. $A_{Y,B+M}(E)\ge S_{Y,B+M}(E)$ for all prime divisors $E$ over $Y$. We keep the notation as above. By construction (see \cite{Kol07}), after possibly replacing the birational morphism $\pi_Y:Y'\to Y$, we may assume that if we write $K_{Y'}+B'+M'\sim_\bQ \pi_Y^*(K_Y+B+M)$, then $M'$ is nef and the coefficient of $E$ in $B'$ is $1-\lct_E(X',G;\phi'^*E)$ where $(X',G)$ is the crepant pullback of $(X,D)$ and the lct is taken only over the generic point of $E$. In particular, $A_{Y,B+M}(E)=\lct_E(X',G;\phi'^*E)$.
Since $(X,(1-\beta)D)$ admits KEE metric for all sufficiently small cone angle $\beta>0$, we have $\delta(X,(1-\beta)D)\ge 1$. In particular, if $(X',G_\beta)$ is the crepant pullback of $(X,(1-\beta)D)$, then $(X',G_\beta + S_{X,(1-\beta)D}(E)\cdot \phi'^*E)$ is lc. Letting $\beta\to 0$ and using \eqref{eq:limS}, we deduce that $(X',G+S_{Y,B+M}(E)\cdot \phi'^*E)$ is lc, hence $S_{Y,B+M}(E)\le \lct_E(X',G;\phi'^*E)=A_{Y,B+M}(E)$ and we obtain $\beta_{Y,B+M}(E)\ge 0$ as desired.
\end{proof}

Unfortunately, the example from Corollary \ref{corollary:delta-upper-bound}(1) shows that only assuming  K-semistability of the base is still not enough for Conjecture \ref{conjecture:CR-conj} to be true. So it seems to to the authors that the existence of KEE metrics on a asymptotically log Fano pair is a subtle problem and the condition $(K_X+D)^{\dim X}=0$ is only necessary. More complicated structures, such as the fibration to the ample model of $-(K_X+D)$, should be taken into consideration.


\begin{ack}
We thank Ivan Cheltsov for suggesting us to write down this 
article. K.F.\ is supported by KAKENHI Grant number 18K13388.
K.Z. is supported by the China post-doctoral grant BX20190014.
\end{ack}

\bibliography{ref}
\bibliographystyle{alpha}







\end{document}